\documentclass{amsart}
\usepackage{amsmath}
\usepackage{amsthm}
\usepackage{listings}
\usepackage{xcolor}
\usepackage{tikz-cd}
\usetikzlibrary{decorations.markings}
\tikzset{negated/.style={
        decoration={markings,
            mark= at position 0.5 with {
                \node[transform shape] (tempnode) {$\backslash$};
            }
        },
        postaction={decorate}
    }
}

\newtheorem{theorem}{Theorem}[section]
\newtheorem{proposition}[theorem]{Proposition}
\newtheorem{lemma}[theorem]{Lemma}
\usepackage{amssymb}

\usepackage{algorithm2e}
\usepackage{enumitem}  
\usepackage{stmaryrd}

\setenumerate{label={\normalfont(\textit{\roman*})}, ref={\textrm{\roman*}},
wide, leftmargin=*, labelwidth=!, labelindent=0pt, 
}

\usepackage{ifthen}
\usepackage{scalerel}
\usepackage{hyperref}
\usepackage{mathrsfs}
\usepackage[mathscr]{euscript}
\usepackage[pagewise]{lineno}

\newtheorem{corollary}[theorem]{Corollary}
\newtheorem*{lemma*}{Lemma}
\newtheorem{alphatheorem}{Theorem}

\theoremstyle{definition}
\newtheorem{definition}[theorem]{Definition}

\newtheorem{remark}[theorem]{Remark}

\newcommand{\rank}{\operatorname{rank}}

\newcommand{\IP}{\mathrm{IP}}

\newcommand{\bra}[1]{\left( #1 \right)}

\renewcommand{\tilde}{\widetilde}
\renewcommand{\bar}{\overline}

\newcommand{\abs}[1]{\left|#1\right|}
\newcommand{\set}[2]{\left\{ #1 \ \middle| \ #2 \right\} }

\newcommand{\e}{\varepsilon}

\newcommand{\NN}{\mathbb{N}}

\newcommand{\QQ}{\mathbb{Q}}

\newcommand{\ZZ}{\mathbb{Z}}

\newcommand{\CC}{\mathbb{C}}

\newcommand{\cA}{\mathcal{A}}

\renewcommand{\subset}{\subseteq}

\newcommand*\patchAmsMathEnvironmentForLineno[1]{\expandafter\let\csname old#1\expandafter\endcsname\csname #1\endcsname
  \expandafter\let\csname oldend#1\expandafter\endcsname\csname end#1\endcsname
  \renewenvironment{#1}{\linenomath\csname old#1\endcsname}{\csname oldend#1\endcsname\endlinenomath}}\newcommand*\patchBothAmsMathEnvironmentsForLineno[1]{\patchAmsMathEnvironmentForLineno{#1}\patchAmsMathEnvironmentForLineno{#1*}}\AtBeginDocument{\patchBothAmsMathEnvironmentsForLineno{equation}\patchBothAmsMathEnvironmentsForLineno{align}\patchBothAmsMathEnvironmentsForLineno{flalign}\patchBothAmsMathEnvironmentsForLineno{alignat}\patchBothAmsMathEnvironmentsForLineno{gather}\patchBothAmsMathEnvironmentsForLineno{multline}}

\begin{document}

\author[J.\ Konieczny]{Jakub Konieczny}
\address[J.\ Konieczny]{Einstein Institute of Mathematics Edmond J. Safra Campus, The Hebrew University of Jerusalem Givat Ram. Jerusalem, 9190401, Israel}
\address{Faculty of Mathematics and Computer Science, Jagiellonian University in Krak\'{o}w, \L{}ojasiewicza 6, 30-348 Krak\'{o}w, Poland}
\email{jakub.konieczny@gmail.com}

\title[Multiplicative automatic sequences]{On multiplicative automatic sequences}

\begin{abstract}
	We show that any automatic multiplicative sequence either coincides with a Dirichlet character or is identically zero when restricted to integers not divisible by small primes. This answers a question of Bell, Bruin and Coons. A similar result was obtained independently by Klurman and Kurlberg.
\end{abstract}


\maketitle 

\section{Introduction}\label{sec:Intro}

Automatic sequences --- that is, sequences computable by finite automata --- give rise to one of the most basic models of computation. As such, for any class of sequences it is natural to inquire into which sequences in it are automatic. In particular, the question of classifying automatic multiplicative sequences has been investigated by a number of authors, including 
 \cite{Schlage-Puchta-2011},
 \cite{BellBruinCoons-2012}, 
 \cite{Hu-2017},
 \cite{AlloucheGoldmakher-2018},
\cite{Li-2019} and
\cite{KlurmanKurlberg-2019}.
The interplay between multiplicative automatic sequences is studied also in 
\cite{Yazdani-2001},
\cite{Schlage-Puchta-2003}, \cite{Coons-2010}, \cite{BellCoonsHare-2014} and
\cite{MullnerLemanczyk-2018}, among others.

The two most recant papers \cite{Li-2019}, \cite{KlurmanKurlberg-2019} listed above give a classification of completely multiplicative automatic sequences, but until now the question remained open for sequences which are multiplicative but not completely so. In particular, the authors of \cite{BellBruinCoons-2012} conjectured that a multiplicative automatic sequence agrees with an eventually periodic sequence on the primes. 

We confirm this conjecture and give some additional structural results. A similar result is also obtained in an upcoming preprint of Klurman and Kurlberg \cite{KlurmanKurlberg-2019-B}.

\begin{alphatheorem}
	If $a \colon \NN_0 \to \CC$ is an automatic multiplicative sequence then there exists a threshold $p_*$ and sequence $\chi$ which is either a Dirichlet character or identically zero such that $a(n) = \chi(n)$ for all $n$ not divisible any prime $p < p_*$.
\end{alphatheorem}

The proof naturally splits into two cases, depending on how often $a$ vanishes. These cases are addressed in Sections \ref{sec:Sparse} and \ref{sec:Dense} respectively.

\begin{remark}
Not all multiplicative sequences satisfying the conclusion of the above theorem are automatic. A full classification of automatic multiplicative sequences appears to still be out of reach of the available techniques, even if barely so. In principle, combining a slightly more precise version of this theorem discussed in subsequent sections and the classification of multiplicative periodic sequences in \cite{LeitmannWolke-1976}, we could completely classify multiplicative $k$-automatic sequences which vanish on all integers not coprime to $k$. The behaviour of $a$ on powers of primes dividing $k$ remains problematic, as evidenced by the fact that when $k$ is prime then for any Dirichlet character with modulus $k^r$ and any root of unity $\xi$, the sequence 
$a_\chi(n) := \xi^{\nu_k(n)} \chi(n/k^{\nu_k(n)})$
is multiplicative and $k$-automatic. The last sequence is a mock Dirichlet character, investigated in \cite{BellBruinCoons-2012}.
\end{remark}

\subsection{Basics and notion}

Throughout, $\NN$ denotes the positive integers and $\NN_0 := \NN \cup \{0\}$. A sequence $a \colon \NN_0 \to \CC$ is multiplicative if $a(nm) = a(n)a(m)$ for any coprime $m,n \in \NN$, and it is completely multiplicative if the assumption of coprimality can be dropped.

For $n \in \NN_0$ we let $[n] = \{1,2,\dots,n\}$ (in particular, $[0] = \emptyset$). If $p$ is a prime, $\alpha \in \NN_0$ and $n \in \NN_0$ then $\nu_p(n)$ denotes the largest power of $p$ which divides $n$ and $p^{\alpha} \parallel n$ means that $\alpha = \nu_p(n)$ (or, equivalently, that $p^\alpha \mid n$ but $p^{\alpha+1} \nmid n$). If $n,m \in \NN$ then $n \perp m$ is shorthand for $\gcd(n,m) = 1$. For two quantities $X$ and $Y$ we write $X = O(Y)$ or $X \ll Y$ if there exists an absolute constant $c$ such that $\abs{X} < cY$. 

 We let $\Sigma_k = \{0,1,\dots,k-1\}$ denote the set of digits in base $k$. For a set $X$ we let $X^*$ denote the set of words over $X$, including the empty word $\epsilon$. If $u = u_1u_2\dots u_l\in \Sigma_k^*$ then $[n]_k = u_1 k^{l-1} + u_2 k^{l-2} + \dots u_l  \in \NN_0$ denotes the integer obtained by interpreting $u$ as a digital expansion in base $k$. Conversely, if $n \in \NN_0$ then $(n)_k \in \Sigma_k^*$ denotes the expansion of $n$ in base $k$ without any leading zeros. More generally, for $l \in \NN_0$ we let $(n)_k^l$ denote the suffix of the word $0^\infty (n)_k$ of length $l$.
 
A sequence $a \colon \NN_0 \to \CC$ is $k$-automatic if there exists finite automaton $\cA = (S,s_0,\delta,\tau)$ which produces $a$. Here, $S$ is a finite set of states, $s_0 \in S$ is the initial state, $\delta$ is the transition function $\Sigma_k \times S \to S$, $(u,s) \mapsto \delta_u(s)$, extended to a map $\Sigma_k^* \times S \to S$ by $\delta_{uv} = \delta_u \circ \delta_v$, $\tau$ is an output function $\tau \colon S \to \CC$, and finally the sequence produced by $\cA$ is $[u]_k \mapsto \tau( \delta_u(s_0))$, where $u \in \Sigma_k$ does not begin with any initial zeros. A sequence is automatic if it is $k$-automatic for some $k \in \NN$.

We fix from now on the automatic multiplicative sequence $a \colon \NN_0 \to \CC$ and an automaton $\cA = (S,s_0,\delta,\tau)$ which produces it. It is well-known that if $k,l \in \NN$ are multiplicatively dependent, i.e., $\log(k)/\log(l) \in \QQ \setminus \{0\}$, then $k$-automatic sequences are the same as $l$-automatic sequences. Hence, we may assume without loss of generality that $k$ is not a perfect power.

\subsection*{Acknowledgements} The author is grateful to Oleksiy Klurman for sharing the aforementioned preprint and for helpful comments and to the anonymous Referee for careful reading of the paper. The authors also wishes to thank Jean-Paul Allouche, Michael Coons and Tamar Ziegler. This research is supported by ERC grant ErgComNum 682150. 
\section{Sparse case}\label{sec:Sparse}

Throughout this section, we make the following assumption:
\begin{equation}\label{eq:ass-sparse}\tag{$\dagger$}
	\text{There exists infinitely many primes } p \text{ such that } a(p^\alpha) =  0 \text{ for some } \alpha \in \NN.
\end{equation}
We let $Z \subset \NN_0$ denote the set of $n \in \NN$ such that $a(n) \neq 0$. It is an automatic set, i.e., a set whose characteristic sequence is automatic.

\begin{proposition}\label{prop:sparse}
	The set $Z$ is  a finite union of (possibly degenerate) geometric progressions with ratio $k^l$ ($l \in \NN_0$).
\end{proposition}

From here, it easily follows that $a(p) = 0$ for large primes. We also note that the fact that $a$ is $k$-multiplicative imposes further restrictions on $Z$ and on the behaviour of $a$ on $Z$. In fact, it is not hard to show that when $k$ is composite, $Z$ needs to be finite. For lack of other nontrivial observations we do not delve further into this subject and devote the remainder of this section to proof of Proposition \ref{prop:sparse}

\subsection{Reduction to arid sets} Our first step is to show that $Z$ is, using the terminology borrowed from \cite{ByszewskiKonieczny-2019}, an arid set.

\begin{definition}
	A set $A \subset \NN_0$ is a \emph{basic arid set} of rank $\leq r$ if it takes the form
\begin{equation}\label{eq:def-of-arid}
		A = \set{ [u_rv_r^{l_r} u_{r-1} \dots u_1 v_1^{l_1} u_0]_k }{ l_1,\dots, l_r \in \NN_0}.
\end{equation}
	for some $u_0,\dots,u_r \in \Sigma_k^*$ and $v_1, \dots, v_r \in \Sigma_k^*$.
	A set $A \subset \NN_0$ is \emph{arid} of rank $\leq r$ if it is a union of finitely many basic arid sets of rank $\leq r$. 
	If $A \subset \NN_0$ then the \emph{rank} of $A$ is the smallest $r$ such that $A$ is contained in an arid set of rank $\leq r$, or $\infty$ if no such $r$ exists.
\end{definition}

\begin{proposition}[{\cite[Proposition 3.4.]{ByszewskiKonieczny-2019}}]\label{prop:arid-dichotomy}
	Let $b \colon \NN_0 \to \CC$ be a $k$-automatic sequence. One of the following is true:
	\begin{enumerate}
	\item \label{it:arid:A}
		The set $\set{n \in \NN_0}{b(n) \neq 0}$ is arid.
	\item \label{it:arid:B}	
		There exists $c \neq 0$ and words $w,v_1,v_2,u \in \Sigma_k^*$ such that $\abs{v_1} = \abs{v_2} > 0$, $v_1 \neq v_2$ and $a([wvu]_k) = c$ for all $v \in \{v_1,v_2\}^*$. 
	\end{enumerate}
\end{proposition}

In the following argument it will be convenient to use the notion of an $\IP_r^+$ set (or a Hilbert cube of dimension $r$), that is, the set of the form
\begin{equation}\label{eq:def-of-IPr}
	A = \set{\textstyle n_0 + \sum_{i \in I} n_i }{ I \subset [r] },
\end{equation}
where $n_0 \in \NN_0$ and $n_1,\dots,n_r \in \NN$. We refer to $n_1,\dots,n_r$ as the \emph{sidelengths} of $A$. 
Note that in condition \ref{prop:arid-dichotomy}(\ref{it:arid:B}), the words $w,u$ can be empty but $v_1,v_2$ cannot, and for each $r \geq 0$ the set $\set{ [wvu]_k}{v \in \{v_1,v_2\}^r}$ is $\IP_r^+$.

\begin{lemma}\label{lem:IPr-CD}
	Let $m \in \ZZ$ and let $A \subset \NN_0$ be an $\IP_r^+$ set with sidelengths coprime to $m$. Then $\#(A \bmod m) \geq \min(m,r+1)$.
\end{lemma}
\begin{proof}
	We proceed by induction on $r$, the case $r = 0$ being trivial. If $r \geq 1$ then we can construe $A$ as the sumset $A' + \{0,n_r\}$ of an $\IP_{r-1}^+$ set $A'$ and a two-element set. Either $\#(A' \bmod m) = m$, in which case we are done, or there exists $n \in A' \bmod m$ such that $n+n_r \not \in A' \bmod{m}$, in which case $\#(A \bmod m) \geq \#(A' \bmod m) + 1 \geq r+1$ so we are also done.
\end{proof}

\begin{proposition}
	The set $Z$ is arid.
\end{proposition}
\begin{proof}
	For the sake of contradiction, suppose that $Z$ was not arid, and let $c$ and $w,v_1,v_2,u$ be as in Proposition \ref{prop:arid-dichotomy}(\ref{it:arid:B}). Assumption \eqref{eq:ass-sparse} guarantees that we can find a prime power $q = p^\alpha$ such that $a(q) = 0$ and $p \nmid [v_1]_k - [v_2]_k$ and $p \nmid k$. Pick $r := p^{\alpha+1}$ and consider the set
	\[
		A := \set{[wvu]_k}{ v \in \{v_1,v_2\}^r }.
	\]
	It follows directly from the defining conditions in \ref{prop:arid-dichotomy}(\ref{it:arid:B}) that $a(n) = c$ for all $n \in A$. On the other hand, $A$ is an $\IP_r^+$ set with sidelengths of the form $([v_1]_k - [v_2]_k)k^{l}$ with $l \in \NN_0$, which are not divisible by $p$. By Lemma \ref{lem:IPr-CD}, $A$ covers all residues modulo $r$. In particular, $A$ contains an integer $n$ exactly divisible by $q$, whence 
	\[ 0 \neq c = a(n) = a(n/q) a(q) = 0,\]
	which is the sought for contradiction.
\end{proof}

We are now left with the task of showing that arid sets cannot be the level sets of multiplicative sequences, except for the arguably trivial cases of geometric progressions whose ratio is a power of $k$. 

\subsection{Base $k$ geometric progressions} 
While arid sets are well-adjusted for studying combinatorial properties of base $k$ expansions, in order to study arithmetic properties it is convenient to work in a slightly more general setup. 
We define a \emph{generalised geometric progression} of rank $\leq r$  as a set of the form
\begin{equation}\label{eq:def-of-ggp}
	A = \set{ x_0 + \sum_{i=1}^r x_i k^{\alpha_i} }{\alpha_1, \dots, \alpha_r \in \NN_0},
\end{equation}
where $x_0,x_1,\dots,x_r \in \QQ$. For the sake of uniformity, define also $\alpha_0 := 0$
Likewise, we define a \emph{restricted generalised geometric progression} of rank $\leq r$ as a set of the form
\begin{equation}\label{eq:def-of-rggp}
	A = \set{x_0 + \sum_{i=1}^r x_i k^{\alpha_i} }{ \alpha_1 \in F_1,\ \alpha_2 \in F_2(\alpha_1), \dots, \alpha_r \in F_r(\alpha_1,\dots,\alpha_{r-1})},
\end{equation}
where $x_1, \dots, x_r \in \QQ$ and $F_i$ are maps $\NN^{i-1}_0 \to \mathscr{P}_{\mathrm{inf}}(\NN_0)$ for each $i \in [r]$. Here, $\mathscr{P}_{\mathrm{inf}}(X)$ denotes the set of all infinite subsets of a set $X$. In a fully analogous manner we define \emph{restricted arid sets} of rank $\leq r$ as sets of the form
\begin{equation}\label{eq:def-of-rarid}
	A = \set{ [u_rv_r^{l_r} u_{r-1} \dots u_1 v_1^{l_1} u_0]_k }{ l_1 \in F_1,\ l_2 \in F_2(l_1), \dots, l_r \in F_r(l_1,\dots,l_{r-1})},
\end{equation}
where $u_0,\dots,u_r \in \Sigma_k^*$ and $v_1, \dots, v_r \in \Sigma_k^*$ and $F_i$ are like above.

Given sequences $F_i$ as in \eqref{eq:def-of-rggp}, let us call a vector $(\alpha_0,\alpha_1, \dots, \alpha_s)$ of length $s \leq r$ \emph{admissible} if $\alpha_0 = 0$ and $\alpha_i \in F_i(\alpha_1,\dots,\alpha_{i-1})$ for all $i \in [s]$, $0 \leq s \leq r$. The elements of the restricted generalised geometric progression $A$ given by \eqref{eq:def-of-rggp} can naturally be indexed by the leaves of a regular rooted tree with vertex degree $\infty$, whose vertices are admissisible sequences $(\alpha_0,\dots,\alpha_s)$, whose root is $(0)$ and whose edges are given by $(\alpha_0,\dots,\alpha_s) \to (\alpha_0,\dots,\alpha_s,\alpha_{s+1})$. By induction on $r$ we see that if the leaves of such a tree are coloured by finitely many colours then there exists an infinite regular subtree of depth $r$ with monochromatic leaves. As a consequence, restricted generalised geometric progressions of a given rank are partition regular. The same observation, mutatis mutandis, applies to restricted arid sets.

It is clear that any (restricted) arid set is a (restricted) generalised geometric progression. The following lemma provides a partial converse to this statement.

\begin{lemma}\label{lem:base-k-exp-of-rggp}
	For any $x_0,x_1,\dots,x_r \in \QQ$ there exists $B \in \NN$ and $C > 0$ such that the following holds. Suppose that $0 = \alpha_0, \alpha_1, \dots, \alpha_r \in \NN_0$ is a sequence such that 
	\begin{equation}\label{eq:209:1}
		x_0 + \sum_{i=1}^r x_i k^{\alpha_i} \in \NN
	\end{equation}
	and $\alpha_{i} \geq \alpha_{i-1} + C$ for all $i \in [r]$. Then there exist words $u_0 \in \Sigma_k^{B}$, $u_1, \dots, u_r \in \Sigma_k^{3B}$ and $v_1,\dots,v_r \in \Sigma_k^B$ such that 
	\begin{equation}\label{eq:209:2}
		x_0 + \sum_{i=1}^r x_i k^{\alpha_i} = [u_r v_{r}^{l_r} u_{r-1} \dots u_1 v_1^{l_1} u_0]_k,
	\end{equation}
	where for $i \in [r]$ the lengths $l_i \in \NN_0$ are uniquely determined by
	\begin{equation}\label{eq:209:3}
		0 \leq \alpha_i - B \bra{\textstyle \sum_{j=1}^i l_j + 3i - 1} < B.
	\end{equation}
	If additionally $x_i \neq 0$ for all $i \in [r]$ then the expansion in \eqref{eq:209:2} is nondegenerate in the sense that $u_r \neq 0^{B}$ and there is no $i \in [r]$ such that $v_i^3 = u_i = v_{i-1}^3$.
\end{lemma}
\begin{proof}
	This follows by inspection of the standard long addition procedure. Suppose first that $x_0,x_1,\dots, x_r$ were all positive integers. Then the conclusion would hold with $v_i = 0^B$ and $u_i = 0^* (x_i)_k 0^*$, where $0^*$ denotes an unspecified string of zeros. If we drop the assumption of positivity, then the same conclusion holds except $v_i$ can also take the form $(k-1)^B$ and $u_i$ need to be modified accordingly. Finally, if $x_i$ are rational then apply the above reasoning to the sequence $M x_0, M x_1, \dots, M x_r$ where $M$ is multiplicatively rich enough that the latter sequence consists of only integers, and use the fact that division by $M$ takes periodic digital expansions to periodic digital expansions. 
\end{proof}
\begin{remark}
\begin{enumerate}[wide]
\item	The definition of $l_i$ in \eqref{eq:209:3} is arranged so that if the right hand side of \eqref{eq:209:2} is construed as the $B$-block expansion of the sum on the left hand side (with each $v_i$ occupying one block and $u_i$ occupying three blocks) then the position $\alpha_i$ falls in the middle block of $u_i$ for all $i \in [r]$.
\item The constant $B$ can be replaced by any multiple, and the constant $C$ can always be enlarged. We could have required that $B = C$, but we believe that would decrease the intuitive appeal of the result.
\end{enumerate}
\end{remark}

Our definition rank guarantees that if $A$ is a (restricted) arid set of rank $\leq r$ then $\rank A \leq r$. It follows from Lemma \ref{lem:base-k-exp-of-rggp} that if $A \subset \NN_0$ is a (restricted) generalised geometric progression of rank $\leq r$ then also $\rank A \leq r$. Below we show that in the situation above we have equality $\rank A = r$, except for some degenerate cases.

\begin{lemma}\label{lem:rank-basic}
	Let $B \in \NN$ and let $u_0 \in \Sigma_k^{B}$, $u_1, \dots, u_r \in \Sigma_k^{3B}$ and $v_1,\dots,v_r \in \Sigma_k^B$ be nondegenerate in the sense of Lemma \ref{lem:base-k-exp-of-rggp}, and let $A$ be the corresponding arid set given by \eqref{eq:def-of-arid}. Then $\rank A = r$.
\end{lemma}
\begin{proof}
	Since $A$ is given by  \eqref{eq:def-of-arid} and nondegenerate in the sense of Lemma \ref{lem:base-k-exp-of-rggp} we have $\# A \cap [N] \gg N^r$ for $N \to \infty$. On the other hand, $\# A \cap [N] \ll N^{\rank A}$, whence $\rank A \geq r$. It remains to recall that also $\rank A \leq r$.
\end{proof}

\begin{remark}
The above lemma can also be derived from the following result, which was used in a previous draft of this paper. We include it since it goes some way towards justifying the many notions of a rank that are implicitly introduced above, but we omit the proof which is technical and not used on the route to the proof of our main result.
\begin{lemma*}\label{lem:rank-strong}
	Let $A$ be a restricted generalised geometric progression of rank $\leq r$ given by \eqref{eq:def-of-rggp} with $x_1,\dots,x_r \neq 0$. Then $\rank A = r$.
\end{lemma*}
\end{remark}

\subsection{Multiplication and arid sets}

Recall that the set $Z$ of nonzero places of $a$ is closed under products of coprime elements. More generally, if $n,m \in Z$ and $d \in \NN$ is such that $d \mid n$, $n/d \perp n$ and $n/d \perp m$ then also $mn/d \in Z$. This motivates the interest in the following lemma.

\begin{lemma}\label{lem:small-gcd-in-ggp}
	For any $u,v,w \in \Sigma_k^*$ with $[u]_k,[w]_k \neq 0$ there exists $D \in \NN$ such that the following is true.
	For any prime $p$ there exists $Q \in \NN$ such that if $l \in \NN$ and $Q \mid l$ then
	\[ 
		\nu_p\bra{[wv^lu]_k} \leq \nu_p(D). 
	\]
\end{lemma}
\begin{proof}
	For reasons which will become clear in the course of the argument, we will take $D := D_0 D_1$ where $D_1 := [wu]_k$ and $D_0 := \abs{ k^{\abs{u}} [v]_k -(k^{\abs{v}}-1) [u]_k}$. Note that $D_1 \neq 0$ since $[w]_k \neq 0$ and $D_0 \neq 0$ since $D_0 \equiv (k^{\abs{v}}-1) [u]_k \not \equiv 0 \bmod{k^{\abs{u}}}$. The argument splits into two cases, depending on whether $p$ divides $k$. Note that in full generality we have
	\[
		[wv^lu]_k = [w]_k k^{l \abs{v} + \abs{u}} + [v]_k \frac{k^{l\abs{v}}-1}{k^{\abs{v}}-1} k^{\abs{u}} + [u]_k.
	\]
	
	Suppose first that $p \nmid k$. For any $\omega \in \NN_0$ there exists $Q_{\omega}$ such that 
	\[
		[wv^lu]_k \equiv [w]_k k^{\abs{u}} + [u]_k = D_1 \pmod{p^\omega}
	\]
	for all $l \in \NN$ divisible by $Q_\omega$. In particular, letting $\omega > \nu_p(D_1)$ we conclude that 
	\[
		\nu_p\bra{ [wv^lu]_k} \leq \nu_p(D_1) \leq \nu_p(D)
	\]
	for all $l$ divisible by $Q_\omega$.
	
	Secondly, suppose that $p \mid k$. Then for any $\omega \in \NN_0$ there exists $Q_\omega$ such that
	\[
		(k^{\abs{v}}-1)[wv^lu]_k \equiv -[v]_k k^{\abs{u}} + (k^{\abs{v}}-1)[u]_k \equiv \pm D_0 \pmod{p^\omega}
	\]
	for all $l \in \NN$ with $l \geq Q_\omega$ (which in particular holds if $Q_\omega \mid l$). Letting $\omega > \nu_p(D_0)$ we conclude that 
	\[
		\nu_p\bra{ [wv^lu]_k} \leq \nu_p(D_0) \leq \nu_p(D)
	\]
	for all $l$ divisible by $Q_\omega$.
\end{proof}

To simplify notation in the following result, for $n,m \in \NN$ let $\gcd(m^\infty,n)$ denote the limit $\lim_{\alpha \to \infty} \gcd(m^\alpha,n)$, or equivalently the product $\prod_{p \mid \gcd(m,n)} p^{\nu_p(n)}$. Note we do not attribute any independent meaning to the symbol $n^\infty$ outside of $\gcd$. It follows directly from Lemma \ref{lem:small-gcd-in-ggp} that, with notation therein, for each $m \in \NN$ there exists an integer $Q$ such that for any $l \in \NN$ divisible by $Q$ we have $\gcd\bra{m^\infty,[wv^lu]_k} \mid D$.

\begin{proposition}\label{cor:small-gcd-in-ggp}
	Let $u,v,w \in \Sigma_k^*$, $v \neq \epsilon$, and $[u]_k,[w]_k \neq 0$ or $[v]_k \neq 0$. Then there exists $l \in \NN_0 $ such that $[wv^lu]_k \not\in Z$.	
\end{proposition}
\begin{proof}
	For the sake of contradiction suppose that $[wv^lu]_k \in Z$ for all $l \in \NN_0$. 
	Replacing $w$ and $u$ with $wv$ and $vu$, we may assume that $[u]_k,[w]_k \neq 0$.
Let $t \in \NN$ be a large parameter. Our strategy is to show that the elements of $Z$ which can be constructed taking products of $t$ terms of the form $[wv^lu]_k$ ($l \in \NN_0$) give rise to an arid set of rank $\geq 2^t -1$, which leads to contradiction since $\rank Z < \infty$.
	
	For $l \in \NN_0$ let $n(l) := [wu^lv]_k$. It follows from Lemma \ref{lem:small-gcd-in-ggp} that there exists $D \in \NN$ such that for any $m \in \NN$ there exist $Q \in \NN$ such that if $l \in \NN$ and $Q \mid l$ then $ \gcd(m^\infty, n(l)) \mid D$. Using this observation iteratively we can find a sequence of infinite sets $F_1$, $F_2(l_1)$, $F_3(l_1,l_2), \dots, F_t(l_1,\dots,l_t) \subset \NN_0$ such that for any sequence $l_1,\dots,l_t \in \NN_0$ with $l_i \in F_i(l_1,\dots,l_{i-1})$ for all $i \in [t]$ we have $d_i := \gcd\bra{\prod_{j=1}^{i-1} n(l_j)^\infty,n(l_i)} \mid D$ for all $i \in [t]$. Using partition regularity, we may further assume that $1 = d_1, d_2, \dots, d_t$ are independent of the choice of $l_1,\dots,l_t$. Hence, for any admissible $l_1,\dots,l_t \in \NN_0$ we have $d_i \mid n(l_i)$, $n(l_i)/d_i \perp d_i$ and $n(l_i)/d_i \perp n(l_j)/d_j$ for all $i,j \in [t]$ with $i \neq j$. Hence, $Z$ contains the set
	\[
		A := \set{\prod_{i=1}^{t} n(l_i)/d_i }{l_1 \in F_1, \dots, l_t \in F(l_1,\dots,l_{t-1})}.
	\]   

	It is clear from the definition of $n(l)$ that there exists $z,y \in \QQ \setminus \{0\}$ and $c \in \NN$ such that $n(l) = z k^{cl} + y$. Hence, for any admissible $l_1,\dots,l_t$ we can expand the product in the definition of $A$ as
\[
	\prod_{i=1}^{t} n(l_i)/d_i = \sum_{I \subset [t]} x_I k^{\alpha_I},
\]	
where $x_I = x^{\abs{I}} y^{t-\abs{I}} \neq 0$ and $\alpha_I = \sum_{i \in I} l_i$ (in particular, $\alpha_{\epsilon} = 0$). Let $s := 2^t-1$. We may identify $\{0,1\}^t$ with $\{0,1,\dots,s\}$ in a standard way. Replacing $F_i$ with smaller sets if necessary we may assume that the sequence $\{\alpha_j\}_{j=0}^{s}$ is increasing for any admissible $l_1,\dots,l_t$, and indeed that $\alpha_j > \alpha_{j-1} + C$ for all $j \in [s]$ where $C > 0$ is an arbitrary constant. In particular, letting $C$ and $B$ be the constants from Lemma \ref{lem:base-k-exp-of-rggp} we conclude that there exists words $u_0 \in \Sigma_k^B$, $u_1, \dots, u_s \in \Sigma_k^{3B}$ and $v_1,\dots,v_s \in \Sigma_k^B$ obeying the nondegeneracy conditions $[u_s] \neq 0$ and $\#\{v_j,u_j^3, v_{j-1}\} > 1$ for all $j \in [s]$, such that 
\[
	\prod_{i=1}^{t} n(l_i)/d_i = x_0 + \sum_{j=1}^s x_j k^{\alpha_j} = [u_s v_{s}^{m_s} u_{s-1} \dots u_1 v_1^{m_1} u_0]_k,  
\]
where $m_1,\dots, m_s \in \NN$ obey the asymptotic relation $m_j = (\alpha_j - \alpha_{j-1})/B + O(1)$ for all $j \in [s]$. Note that we can assume that $u_j,v_j$ are independent of $l_1,\dots,l_t$ by partition regularity. By the same token, we may also assume that $\bar m_j := m_j \bmod M$ ($j \in [s]$) are independent of $l_1,\dots,l_t$, where $M$ is a multiplicatively rich constant such that $\delta_v^M$ is idempotent for each $v \in \Sigma_k^*$ (we can take $M = \#{S}!$). If now follows that $Z$ contains the arid set
\[
	\set{[u_s v_{s}^{m_s'} u_{s-1} \dots u_1 v_1^{m_1'} u_0]_k}{ m_j' \in \bar{m}_j + M \NN \text{ for all } j \in [s]},
\]
whose rank is equal to $s$. In particular, $\rank Z \geq s$, as needed.
\end{proof}

\begin{corollary}
	The set $Z$ is a finite union of geometric progressions of the form $\set{ x k^{cl} }{ l \in \NN_0}$ where $x \in \NN_0$ and $c \in \NN_0s$.
\end{corollary}	
\begin{proof}
	It is enough to notice that the only basic arid sets not containing patterns forbidden by Proposition \ref{cor:small-gcd-in-ggp} take the form $\set{[w0^{cl}]}{ l \in \NN_0}$ with $c \in \NN_0$.
\end{proof}

\section{Dense case} \label{sec:Dense}

We now assume that there exists a threshold $p_0$ such that the following holds:
\begin{equation}\label{eq:ass-dense}\tag{$\ddagger$}
	\text{ For all primes } p \geq p_0 \text{ and all } \alpha \in \NN \text{ we have } a(p^\alpha) \neq  0.
\end{equation}
Our main aim is to show that $a(n)$ coincides with a Dirichlet character for $n$ devoid of small prime factors. We also record some observations concerning the behaviour of $a$ on small primes.

\subsection{Large primes}
We first deal with large primes. From this point onwards, we let $m$, $\chi$ and $p_1$ denote the objects in the following result; we may assume that $p_1 > m$ and that $p_1 > k$.

\begin{proposition}\label{prop:dense-large}
	There exists a Dirichlet character $\chi$ with modulus $m$ and a threshold $p_1$ such that $a(n) = \chi(n)$ for all $n \in \NN$ which are products of primes $\geq p_1$.  
\end{proposition}

Relying on the following result, we can prove Proposition \ref{prop:dense-large} by essentially the same methods which were used by Klurman and Kurlberg \cite{KlurmanKurlberg-2019} for completely multiplicative sequences.

\begin{proposition}\label{prop:dense-large-totmult}
	There exists a threshold $p_2$ such that if $p \geq p_2$ is a prime then $a(p^\alpha) = a(p)^\alpha \neq 0$ for all $\alpha \in \NN$.
\end{proposition}

A crucial step in the argument is a theorem due to Elliott and Kish. 
\begin{theorem}{{\cite[Thm.\ 2]{ElliotKish-2017}}}\label{thm:EK}
Let $A,B,C,D$ be integers with $A,C > 0$ and $\Delta = AD-BC \neq 0$, and let $b \colon \NN_0 \to \CC$ be a completely multiplicative sequence. Put $M = 6 \gcd(A,C) A^2 C^2 \Delta^3$. Suppose that there exists $z \in \CC \setminus \{0\}$ such that that
\[
	b\bra{\frac{An+B}{Cn+D}} = z \text{ for all sufficiently large } n \in \NN_0.
\]
Then there exists a Dirichlet character $\chi$ modulo $M$ such that $b(n) = \chi(n)$ for all $n \in \NN_0$ coprime to $M$. 
\end{theorem}
 We will also need the fact that the $k$-{kernel} of any $k$-automatic sequence is finite. Here, the $k$-\emph{kernel} of $a$ is the set of all the sequences $n \mapsto a(k^\alpha n + r)$ with $\alpha \in \NN_0$ and $0 \leq r < k^\alpha$. 

\begin{proof}[Proof of Prop.\ \ref{prop:dense-large} assuming Prop.\ \ref{prop:dense-large-totmult}]
	Suppose for the sake of clarity that $p_2 \geq p_0$. Because the $k$-kernel of $a$ is finite, we can find integers $\beta < \gamma$ such that $a(k^\beta n + 1) = a(k^\gamma n + 1)$ for all $n \in \NN$.
	 Let $\tilde a$ denote the multiplicative function given by $\tilde a(p^\alpha) = a(p^\alpha) = a(p)^\alpha$ if $p \geq p_2$ and $\tilde a(p^\alpha) = 1$ if $p < p_2$, $\alpha \in \NN$. Clearly, $\tilde a$ is totally multiplicative. Because $a$ takes on only finitely many values, $a(p)$ is a root of unity for each $p \geq p_2$. Letting $Q$ be the product of all primes $< p_2$ we obtain 
	\[
		\tilde a(k^\beta Q n + 1) = a(k^\beta Q n + 1) = a(k^\gamma Q n + 1) = \tilde a(k^\gamma Q n + 1) \neq 0
	\]
	for all $n \in \NN_0$. It now follows from Theorem \ref{thm:EK} that $\tilde a$ coincides with a Dirichlet character on all powers of large primes.
\end{proof}

In order to prove Proposition \ref{prop:dense-large-totmult}, it will be convenient to introduce an equivalence relation on $\Sigma_k^*$ where $u \sim v$ if $\abs{u} = \abs{v}$ and words $u$, $v$ give rise to the same transition function in the automaton $\cA$, that is, $\delta_u = \delta_v$. Since transition functions are self-maps of $S$, the number of equivalence classes $\#\bra{ \Sigma_k^l/{\sim}} $ is bounded uniformly with respect to $l \in \NN$. Likewise, consider the equivalence relation on $\NN_0$ given by $n_1 \sim n_2$ if $(n_1)_k^l \sim (n_2)_k^l$ for all sufficiently large $l \in \NN$, or --- equivalently --- if $(n_1)_k^l \sim (n_2)_k^l$ for at least one $l \in \NN$ with $n_1, n_2 < k^l$. Crucially, there are only finitely many equivalence classes: $\#\bra{ \NN_0/{\sim}} < \infty$.

\begin{lemma}\label{lem:equiv-pairs}
	There exists a threshold $p_3$ such that for any $p > p_3$ there exists a pair $n_1,n_2 \in \NN_0$ with $n_1 \not \equiv n_2 \pmod{p}$ such that $n_1 \sim n_2$ and $pn_1 \sim pn_2$.
\end{lemma}
\begin{proof}
	Since $\#\bra{ \NN_0/{\sim}} < \infty$, this follows from the pidgeonhole principle.
\end{proof}

For the sake of brevity, in the following argument and elsewhere we will say that a statement $\varphi(n)$ is true for \emph{almost all} $n$ if the set of $n$ for which it fails has asymptotic density $0$:
\[
	\lim_{N \to \infty} \frac{1}{N} \# \set{ n \in [0,N)} { \neg \varphi(n) } = 0.
\]

\begin{proof}[Proof of Prop.\ \ref{prop:dense-large-totmult}]
	Take any prime $p$ with $p > p_3$, where $p_3$ is the threshold in Lemma \ref{lem:equiv-pairs}, and let $n_1,n_2$ be the pair whose existence is guaranteed by said lemma. Let $l$ be a large integer, to be determined in the course of the argument, and put $u_i := (n_i)_k^l$, $u_i' := (pn_i)_k^l$. It is a well-known fact that if $w \in \Sigma_k^*$ then for almost all $n$ the expansion $(n)_k$ contains $w$. Hence, for almost all $n$ there exists a decomposition
\[
	(n)_k = x_n u_1 y_n,
\]
for some $x_n, y_n \in \Sigma_k^*$ where $x_n$ is nonempty and does not start with any zeros and $y_n$ starts with at least $l$ zeros. Letting $x'_n$ and $y'_n$ denote the expansions of $p[x_n]_k$ and $p[y_n]_k$ with $x'_n$ not starting with any zeros and $\abs{y'_n} = \abs{y_n}$, we get the decomposition
\[
	(pn)_k = x_n' u_1' y_n'.
\]
If $p \nmid n$ then clearly $a(pn) = a(p)a(n)$. On the other hand, if $p \mid n$ then
\begin{align*}
	a(pn) &= a([x_n' u_1' y_n']_k)  
	= a([x_n' u_2' y_n']_k) 
	\\& = a(p) a([x_n u_2 y_n]_k)
	= a(p) a([x_n u_1 y_n]_k) = a(p) a(n).
\end{align*}
Hence, we have shown that $a(pn) = a(p) a(n)$ for almost all $n$.

Let $\alpha \in \NN$. Integers $n$ such that $p^\alpha \parallel n$ and $n \perp q$ for all $q < p_0$ constitute a positive proportion of all integers, whence there exists many $n$ such that
\[
	a(p^{\alpha+1}) = \frac{a(pn)}{a(n/p^\alpha)} = \frac{a(p)a(n) }{ a(n/p^\alpha) } = 
	a(p)a(p^\alpha).
\]
It now follows by induction that $a(p^\alpha) = a(p)^\alpha$.	 
\end{proof}

\subsection{Small primes} 

In this section we address the behaviour of $a$ on small primes. Unfortunately, we can only obtain a weaker analogue of Proposition \ref{prop:dense-large}.

\begin{proposition}\label{prop:dense-small-prime}
	For any prime $p \nmid k$, the sequence $a(p^\alpha)$ is eventually periodic.
\end{proposition}
\begin{proof}
	Recall that there exists (many) pairs of distinct integers $n_1,n_2 \in \NN_0$ such that $n_1 \sim n_2$. Note also that if $n_1 \sim n_2$ and $n_1' \sim n_2'$ then also $k^\alpha n_1 + n_1' \sim k^\alpha n_2 + n_2'$ for sufficiently large $\alpha$. Hence, we can assume that $d := n_1 - n_2$ is divisible by any prime $p < p_1$ and also by a large power of $k$. Let $v_1 = (n_1)_k^l$ and $v_2 = (n_2)_k^l$ where $l$ is a large integer. 
		
	For any $\alpha \in \NN$ and any $\beta$ sufficiently large in terms of $\alpha$, there exists a prime $q$ such that $(q p^\alpha)_k \in 1v_1 \Sigma_k^\beta$, that is, the expansion of $q p^\alpha$  starts with $1v_1$ and contains $\beta$ other digits. (This follows from the classical fact for any $\e > 0$ and any sufficiently large $N$, there exists a prime between $N$ and $N+\e N$; in fact, by the Prime Number Theorem there are  $\sim \e N/\log N$ such primes.) Let $\delta = \nu_p(d)$ and suppose that $\alpha > \delta$. Then 
\begin{align*}
	a(p^\alpha) &= a(qp^\alpha)/a(q) = a(qp^\alpha + d k^\beta)/a(q)
	\\ &= a(qp^{\alpha-\delta} + d k^\beta/p^\delta)a(p^\delta)/a(q)
	= \chi(qp^{\alpha-\delta} + d k^\beta/p^\delta)a(p^\delta)/\chi(q),
\end{align*}  
where in the last transition we use the fact that any prime $< p_1$ divides exactly one of  $qp^{\alpha-\delta}$ and $(d/p^\delta) k^\beta$. We may also assume (using the Prime Number Theorem in arithmetic progressions) that $q \equiv 1 \bmod{m}$ and $d k^\beta/p^\delta \equiv d/p^\delta \bmod{m}$, whence 
\begin{align*}
	a(p^\alpha) = \chi(p^{\alpha-\delta} + d/p^\delta)a(p^\delta).
\end{align*}  
It remains to notice that the expression on the right hand side is periodic in $p^\alpha$.
\end{proof}

\begin{corollary}
	There exists a periodic sequence $b \colon \NN_0 \to \CC$ and threshold $n_0$ such that $a(n) = b(n)$ for all $n \geq n_0$ coprime to $k$.
\end{corollary}
\begin{proof}
	Partitioning $\NN_0$ into arithmetic progressions, we may assume that for each prime $p < p_1$, either $n$ is divisible by a large power of $p$ or $n$ is not divisible by $p$. Repeating the same reasoning as in Proposition \ref{prop:dense-small-prime} we conclude that
	\begin{align*}
	a(n) = \chi\bra{ \textstyle(n + d)/\prod_{p} p^{\delta_p} } \prod_{p} a\bra{p^{\delta_p}},
\end{align*} 
where $\delta_p = \nu_p(d)$ and the product runs over all primes $p < p_1$ with $p \nmid k$ and $p \mid n$.
\end{proof}

\bibliographystyle{alphaabbr}
\bibliography{bibliography}

\end{document}